\documentclass[11pt,reqno]{amsart}
\usepackage{amsthm}
\usepackage{amssymb}
\usepackage{latexsym}
\usepackage{multicol}
\usepackage{verbatim,enumerate}
\usepackage{amscd}
\usepackage{bbm}
\usepackage{mathdots}

 \usepackage[usenames]{color}
\usepackage{hyperref}
\usepackage{amsmath, amscd}

\advance\textwidth by 1.2in \advance\oddsidemargin by -.6in \advance\evensidemargin by -.6in
\theoremstyle{definition}
\newtheorem{defn}{Definition}[section]
\newtheorem{thm}{Theorem}[section]
\newtheorem*{thm*}{Theorem}

\newtheorem{rem}{Remark}[section]

\newtheorem{coro}{Corollary}[section]
\newtheorem{lem}{Lemma}[section]
\newtheorem{prop}{Proposition}[section]

\newcounter{cnt}
 \makeatletter
\def\mydggeometry{\makeatletter\dg@YGRID=1\dg@XGRID=20\unitlength=0.003pt\makeatother}
\makeatother \theoremstyle{remark}

\numberwithin{equation}{section}

\let\bwdg\bigwedge
\def\bigwedge{{\textstyle\bwdg}}

\newcommand{\wt}{\operatorname{wt}}

\newcommand{\nc}{\newcommand}
\newcommand{\rnc}{\renewcommand}

\nc{\cal}{\mathcal} \nc{\goth}{\mathfrak} \rnc{\bold}{\mathbf}

\renewcommand{\Bbb}{\mathbb}
\nc\bomega{{\mbox{\boldmath $\omega$}}} \nc\bpsi{{\mbox{\boldmath $\Psi$}}}
 \nc\balpha{{\mbox{\boldmath $\alpha$}}}
 \nc\bpi{{\mbox{\boldmath $\pi$}}}
\nc\bmu{{\mbox{\boldmath $\mu$}}} \nc\bcN{{\mbox{\boldmath $\cal{N}$}}} \nc\bcm{{\mbox{\boldmath $\cal{M}$}}} \nc\blambda{{\mbox{\boldmath
$\lambda$}}}\nc\bnu{{\mbox{\boldmath $\nu$}}}

\newcommand{\lie}[1]{\mathfrak{#1}}

\makeatletter
\def\section{\def\@secnumfont{\mdseries}\@startsection{section}{1}%
  \z@{.7\linespacing\@plus\linespacing}{.5\linespacing}%
  {\normalfont\scshape\centering}}
\def\subsection{\def\@secnumfont{\bfseries}\@startsection{subsection}{2}%
  {\parindent}{.5\linespacing\@plus.7\linespacing}{-.5em}%
  {\normalfont\bfseries}}
\makeatother

 \nc{\Hom}{\operatorname{Hom}}
 \nc{\Ind}{\operatorname{Ind}}

\nc{\N}{{\Bbb N}} \nc\boa{\bold a} \nc\bob{\bold b} \nc\boc{\bold c} \nc\bod{\bold d} \nc\boe{\bold e} \nc\bof{\bold f} \nc\bog{\bold g}
\nc\boh{\bold h} \nc\boi{\bold i} \nc\boj{\bold j} \nc\bok{\bold k} \nc\bol{\bold l} \nc\bom{\bold m} \nc\bon{\bold n} \nc\boo{\bold o}
\nc\bop{\bold p} \nc\boq{\bold q} \nc\bor{\bold r} \nc\bos{\bold s} \nc\boT{\bold t} \nc\boF{\bold F} \nc\bou{\bold u} \nc\bov{\bold v}
\nc\bow{\bold w} \nc\boz{\bold z} \nc\boy{\bold y} \nc\ba{\bold A} \nc\bb{\bold B} \nc\bc{\bold C} \nc\bd{\bold D} \nc\be{\bold E} \nc\bg{\bold
G} \nc\bh{\bold H} \nc\bi{\bold I} \nc\bj{\bold J} \nc\bk{\bold K} \nc\bl{\bold L} \nc\bm{\bold M} \nc\bn{\bold N} \nc\bo{\bold O} \nc\bp{\bold
P} \nc\bq{\bold Q} \nc\br{\bold R} \nc\bs{\bold S} \nc\bt{\bold T} \nc\bu{\bold U} \nc\bv{\bold V} \nc\bw{\bold W} \nc\bz{\mathbb{ Z}} \nc\bx{\bold
x} \nc\KR{\bold{KR}} \nc\rk{\bold{rk}} \nc\het{\text{ht }}
\nc\toa{\tilde a} \nc\tob{\tilde b} \nc\toc{\tilde c} \nc\tod{\tilde d} \nc\toe{\tilde e} \nc\tof{\tilde f} \nc\tog{\tilde g} \nc\toh{\tilde h}
\nc\toi{\tilde i} \nc\toj{\tilde j} \nc\tok{\tilde k} \nc\tol{\tilde l} \nc\tom{\tilde m} \nc\ton{\tilde n} \nc\too{\tilde o} \nc\toq{\tilde q}
\nc\tor{\tilde r} \nc\tos{\tilde s} \nc\toT{\tilde t} \nc\tou{\tilde u} \nc\tov{\tilde v} \nc\tow{\tilde w} \nc\toz{\tilde z} \nc\woi{w_{\omega_i}}

\nc\BT{\mathbf{T}}
\nc\hei{\operatorname{ht}}
\nc\chara{\operatorname{Char}}
\nc{\buf}{\underline{\bof}} 
\nc{\avee}{\alpha^{\vee}}
\nc{\bLambda}{\mathbf{\Lambda}}
\nc{\bMu}{\mathbf{M}}
\nc{\bNu}{\mathbf{N}}
\nc{\btau}{\mathbf{\sigma}}
\nc{\bsigma}{\mathbf{\sigma}}

\begin{document}

\author[Fourier]{Ghislain Fourier}
\address{Mathematisches Institut, Universit\"at zu K\"oln, Germany}
\email{gfourier@math.uni-koeln.de}
\address{School of Mathematics and Statistics, University of Glasgow, UK}
\email{ghislain.fourier@glasgow.ac.uk}

\thanks{The author was partially supported by the DFG priority program 1388 ``Representation Theory``}

\date{\today}

\title[Homogeneous ideals]{New homogeneous ideals for current algebras: filtrations, fusion products and Pieri rules}
\begin{abstract} 
New graded modules for the current algebra of $\lie{sl}_n$ are introduced. Relating these modules to the fusion product of simple $\lie{sl}_n$-modules and local Weyl modules of truncated current algebras shows their expected impact on several outstanding conjectures. We further generalize results on PBW filtrations of simple $\lie{sl}_n$-modules and use them to provide decomposition formulas for these new modules in important cases.
\end{abstract}
\maketitle

\section{Introduction}
We consider the simple complex Lie algebra $\lie{sl}_n = \lie{b} \oplus \lie{n}^-$ and its current algebra $\lie{sl}_n \otimes \bc[t]$. We fix a pair $(\lambda_1, \lambda_2)$ of dominant integral $\lie{sl}_n$-weights. $F_{\lambda_1, \lambda_2}$ will be introduced as the cyclic $\lie{sl}_n \otimes \bc[t]$-module defined by the homogeneous ideal generated by the kernel of an evaluation map of $\lie{b} \otimes \bc[t]$ and certain monomials in $U(\lie{n}^- \otimes \bc[t])$. $F_{\lambda_1, \lambda_2}$ decomposes into simple, finite-dimensional $\lie{sl}_n$-modules:
\[
F_{\lambda_1, \lambda_2} =_{\lie{sl}_n} \bigoplus_{\tau \in P^+} V(\tau)^{\oplus a_{\lambda_1, \lambda_2}^{\tau}}.
\]
As $F_{\lambda_1, \lambda_2}$ is a highest weight module, we have $a_{\lambda_1, \lambda_2}^{\lambda_1 + \lambda_2} = 1 \text{ and } a_{\lambda_1, \lambda_2}^{\tau} = 0 \text{ if } \tau \nleq \lambda_1 + \lambda_2.$ 
Moreover, $\lie{sl}_n \otimes t^2 \bc[t].F_{\lambda_1, \lambda_2} = 0$ and hence
\[
F_{\lambda_1, \lambda_2} =  U(\lie{n}^- \otimes \bc[t]/(t^2)).\mathbbm{1} \cong U(\lie{n}^-) S(\lie{n}^-).\mathbbm{1}.
\] 
Due to this observation, the $\lie{sl}_n$-highest weight vectors and therefore the multiplicities $ a_{\lambda_1, \lambda_2}^{\tau}$ should be ``controlled`` by $S(\lie{n}^-).\mathbbm{1}$. This provides a close relation to the framework of PBW filtrations (\cite{FFoL11a, FFoL13a}).  By construction, $F_{\lambda_1, \lambda_2}$ is a quotient of $S(\lie{n}^-)/\mathcal{I}(\lambda_1, \lambda_2)$ with an induced $\lie{n}^+$-action $\circ$, where the ideal is generated by 
\[
U(\lie{n}^+) \circ \langle f_\alpha^{a_\alpha + 1}\, | \, \text{ for all positive roots } \alpha \rangle \subset S(\lie{n}^-)
\]
for some $a_\alpha$ depending on $\lambda_1, \lambda_2$. Generalizing the results from \cite[Theorem  and Theorem B]{FFoL11a} we see that a spanning set of $S(\lie{n}^-).\mathbbm{1}$ can be parameterized by integer points in a polytope defined through Dyck paths conditions (Corollary~\ref{span-m}). This leads to the question wether one can give a polytope parametrizing the highest weight vectors. We give a positive answer in certain important cases:
\begin{thm*}\label{intro-thm} Suppose $\lambda_1, \lambda_2 $ satisfy one of the following:
\begin{enumerate}
\item $\lambda_1, \lambda_2$ are both rectangular weights, e.g. multiples of some fundamental weights $\omega_i, \omega_j$,
\item $\lambda_1$ is arbitrary and $\lambda_2$ is either $\omega_j$ or $k \omega_1$,
\item $\lambda_1 + w(\lambda_2)$ is dominant for all Weyl group elements,
\end{enumerate}
then for all dominant weights $\tau$:
\[  a_{\lambda_1, \lambda_2}^{\tau} =  c_{\lambda_1, \lambda_2}^{\tau}, \text{  the Littlewood-Richardson coefficients.}\]
\end{thm*}
\noindent
Part (2) might be seen as a Pieri rule while part (3) covers $\lambda_1 \gg \lambda_2$. So for fixed $\lambda_2$ we cover the minimal case, e.g. $\lambda_1$ being a fundamental weight, and the large case, e.g. $\lambda_1 \gg \lambda_2$. Note that the results from \cite{CV13}, \cite{V13} imply $a_{\lambda_1, \lambda_2}^{\tau} =  c_{\lambda_1, \lambda_2}^{\tau}$ for all $\tau \in P^+$ if $\lambda_1 = m \omega_i$ and the height of $\lambda_2$ is less than $m+1$. This covers of course part (1) of the theorem but we provide here a different proof using the relation to the PBW filtration.\\

\noindent
The paper is motivated by the search for homogeneous ideals in $U(\lie{sl}_n \otimes \bc[t])$ defining the fusion product $V(\lambda_1)_{c_1} \ast V(\lambda_2)_{c_2}$ of two simple $\lie{sl}_n$-modules. This is the associated graded module of the tensor product of corresponding evaluation modules (\cite{FL99} and also Section~\ref{fusion}). These ideals can be deduced straightforward for $\lie{sl}_2$ (Lemma~\ref{fusion-sl2}) and we generalize this to obtain generators for every $\lie{sl}_2$-triple. The theorem implies that in the considered cases (Lemma~\ref{whatever})
\[
F_{\lambda_1, \lambda_2} \cong V(\lambda_1)_{c_1} \ast V(\lambda_2)_{c_2}
\]
and we conjecture that this is true for all pairs of dominant integral weights.\\ 
Let us briefly explain why these modules $F_{\lambda_1, \lambda_2}$ and especially the conjectured isomorphism to the fusion product is of special interest. In fact, this is closely related to several important conjectures:\\
The first one is the conjecture that the fusion product of finitely many tensor factors is independent of the evaluation parameter (\cite{FL99}). This independence has been proved for some classes of modules but so far not for arbitrary tuples of dominant integral weights. Note that two-factor case can be deduced from straightforward calculations.\\
The second conjecture is on Schur positivity of certain symmetric functions. In \cite{CFS14} (see also \cite{DP2007}) a partial order on pairs of dominant weights has been introduced. It is conjectured that along with the partial order, the difference of the products of the corresponding Schur functions is a non-negative linear combination of Schur functions (this has been conjectured also independently by Lam, Postnikov and Pylyavskyy), hence \textit{Schur positive}. Note here, that this generalizes a conjecture on Schur positivity along row shuffles  (\cite{Oko97, FFLP05}), proved in \cite{LPP07}.\\
The third related conjecture is on local Weyl modules for truncated current algebras. Local Weyl modules for generalized current algebras, $\lie{sl}_n \otimes A$, where $A$ is a commutative, associative, unital $\bc$-algebra, have gained much attention in the last two decades. Due to their homological properties they play an important role in the category of finite-dimensional $\lie{sl}_n \otimes A$-modules, which is not semi-simple in general (for more see \cite{CFK10}).\\
Although quite a lot of research has been done on local Weyl modules, their explicit character is known for a few algebras only. For $A = \bc[t^{\pm 1}], \bc[t]$ their character is given by the tensor product of fundamental modules for $\lie{sl}_n$ (\cite{CL06, FoL07}), for semi-simple, finite-dimensional $A$, the character is given by $\dim A$ copies of a simple $\lie{sl}_n$-module. Besides these cases the character is not known for general local Weyl modules,  not even for the ``smallest`` non-semi-simple algebra $A = \bc[t]/(t^2)$.\\
It is conjectured that for $A = \bc[t]/(t^K)$ this character is also, similar to $\bc[t]$, given by the tensor product of simple $\lie{sl}_n$-modules. We investigate here on the $K=2$ case and prove that in this case the local Weyl modules are isomorphic to certain $F_{\lambda_1, \lambda_2}$, more detailed: $(\lambda_1, \lambda_2)$ is the unique maximal element in the aforementioned poset of pairs of dominant weights (adding up to a fixed $\lambda$). \\
A proof, that the $a_{\lambda_1, \lambda_2}^{\tau}$ are in fact the Littlewood-Richardson coefficients, would imply the conjectures on Schur positivity and on local Weyl modules immediately (Lemma~\ref{fusion-sp}, Lemma~\ref{weyl-f}) and gives another proof for the two-factor of the independence conjecture (Lemma~\ref{whatever}).

The paper is organized as follows: In Section~\ref{one} we refer to the basic definitions and in Section~\ref{graded} we introduce the modules $F_{\lambda_1, \lambda_2}$, proving first properties. In Section~\ref{PBW} we recall the PBW filtration and work out the relation to our new modules, while in Section~\ref{fusion} we recall the fusion products and work out their relation to our modules. In Section~\ref{firstproof} we give the proof for the $\lie{sl}_2$ -case and part(3) of Theorem~\ref{intro-thm}. Section~\ref{KR} contains the proofs of part (1) of Theorem~\ref{intro-thm} and Section~\ref{pieri} the proof of part (2). Section~\ref{poset} recalls the partial order on pairs of dominant weights and also local Weyl modules, and relates these constructions to the new modules.

\textbf{Acknowledgement:} The author would like to thank Evgeny Feigin for various discussions on these modules and explaining the calculations for the independence conjecture in the two-factor case, and further Christian Korff for asking about Pieri rules.

\section{Preliminaries}\label{one}
 Let $\lie{g} = \lie {sl}_{n}(\bc)$, the special linear Lie algebra. We fix a triangular decomposition $\lie{sl}_n = \lie n^+ \oplus \lie h \oplus \lie n^-$ and denote a fixed set of simple roots $\Pi = \{ \alpha_1, \ldots, \alpha_{n-1}\}$, here we use the numbering from \cite{Bou02}. Further, we set $I =\{ i, \ldots, n-1 \}$. The sets of roots is denoted R, the set of positive roots $R^+$. Every root $\beta \in R^+$ can be expressed uniquely as $\alpha_i + \alpha_{i+1} + \ldots + \alpha_j$ for some $i \leq j$, we denote this root $\alpha_{i,j}$.
For $\alpha \in R$, we denote the root space 
\[\lie g_{\alpha} = \{ x \in \lie{sl}_n \, | \, [h,x] = \alpha(h) x \; \forall \; h \, \in \lie h \} = \langle x_\alpha^+ \, | \, \text{ if } \alpha \in R^+ \rangle.\] 
Further, for $\alpha \in R^+$, we fix a $\lie{sl}_2$-triple $\{ x_\alpha^+, x_{\alpha}^-, h_{\alpha} \}$. Denote $P \subset \lie h^*$, respective $P^+$ the integral weights, respective dominant integral weights, and $\{ \omega_1, \ldots, \omega_{n-1}\}$ the set of fundamental weights.\\

\subsection{}
We recall some notations and facts from representation theory. Let $V$ be a finite-dimensional $\lie {sl}_{n}$-module, then $V$ decomposes into its weight spaces with respect to the $\lie h$-action
\[
V = \bigoplus_{ \tau \in P } V_\tau =  \bigoplus_{ \tau \in P }  \{ v \in V \; | \; h.v = \tau(h).v \text{ for all } h \in \lie h\}
\]
\noindent
$P^+$ parameterizes the simple finite-dimensional modules. For $\lambda \in P^+$ we denote the simple, finite-dimensional $\lie {sl}_n$-module of highest weight $\lambda$ by $V(\lambda)$. Further we denote by $v_{\lambda}$ a highest weight vector of $V(\lambda)$

The category of finite-dimensional $\lie{sl}_n$-modules is semi-simple, hence the tensor product of two simple modules decomposes into the direct sum of simple modules, so for $\lambda_1, \lambda_2 \in P^+$
\[
V(\lambda_1) \otimes V(\lambda_2) \cong_{\lie{sl}_n} \bigoplus_{\tau \in P^+} V(\tau)^{c_{\lambda_1, \lambda_2}^\tau}.
\]
Here, $c_{\lambda_1, \lambda_2}^\tau$ denotes the multiplicity of the simple module $V(\tau)$ in a decomposition of the tensor product. These numbers are known as Littlewood-Richardson coefficients and there are several known formulas to compute them (\cite{Kli68, N2002, Lit94} to name but a few).

\subsection{}
The vector space $\lie {sl}_n \otimes \bc[t]$ equipped with the bracket 
\[
[x \otimes p(t) , y \otimes q(t)] = [x,y] \otimes p(t)q(t) \; \; \forall \, x,y \in \lie{sl}_n, p(t), q(t) \in \bc[t]
\]
is a Lie algebra and called the \textit{current algebra} of $\lie {sl}_n$. One may also view this as the Lie algebra of regular functions on $\bc$ with values in $\lie {sl}_n$ (see \cite{NSS12}). The natural grading on $\bc[t]$ induces a grading on $U(\lie{sl}_n \otimes \bc[t])$, where the component of degree $0$ is $U(\lie{sl}_n \otimes 1)$.

\noindent
For a fixed $k \geq 1$, the \textit{truncated current algebra} is the graded quotient of the current algebra 
\[
\lie {sl}_n \otimes \bc[t] / (\lie{sl}_n \otimes t^K \bc[t]) \cong \lie{sl}_n \otimes \bc[t]/(t^K).
\]
 In this paper we will be dealing mainly with the $K=2$ case. Then $U(\lie{sl}_n \otimes \bc[t]/(t^2))$ can be seen as the smash product of $U(\lie{sl}_n)$ and the polynomial ring $S(\lie n^-)$ (\cite{Hag13}).

\subsection{}
The representation theory of $\lie{sl}_n \otimes \bc[t]$ has been subject to a lot of research during the last 25 years (\cite{CP01, FF02, CM04, CL06, FoL06, FoL07, Nao12} to name but a few). The most important property we should mention is, that the category of finite-dimensional $\lie{sl}_n \otimes \bc[t]$-modules is not semi-simple. \\
Every simple, finite-dimensional module is the tensor product of evaluation modules (\cite{Rao93}). This is still true if we replace $\bc[t]$ by a commutative, finitely generated algebra $A$ and instead of complex numbers, evaluations in pairwise distinct maximal ideals \cite{CFK10}.\\ 
Although the simple, finite-dimensional modules are therefore easily described and quite well understood, the task of understanding the indecomposable modules is still unsolved (besides the case $A = \bc[t], \bc[t^{\pm}]$ and the cases where $A$ is finite-dimensional and semi-simple).\\ 
Even in the case where $A$ is the two-dimensional truncated polynomial ring, $A = \bc[t]/(t^2)$, the category of finite-dimensional modules is far from being well understood. While the simple modules are in one-to-one correspondence to simple modules of $\lie{sl}_{n}$ (by using the evaluation at the unique maximal ideal of $ \bc[t]/(t^2)$, \cite{CFK10}), there is not much known about indecomposables, projectives etc. We will return to this point in Section~\ref{poset}.


\section{Some new graded module}\label{graded}

We introduce new graded modules for $\lie {sl}_{n}  \otimes \bc[t]$ as follows. For fixed $\lambda_1, \lambda_2 \in P^+$, let 
\[
(\lambda_1 + \lambda_2) : \lie h \longrightarrow \bc_{\lambda_1 + \lambda_2}
\]
be the one-dimensional $\lie h$-module. We extend this trivially to an action of $\lie b = \lie n^+ \oplus \lie h$ on $\bc_{\lambda_1 + \lambda_2}$. And further, by evaluation at $t = 0$, we obtain a one-dimensional module
\[
\lie b \otimes \bc[t] \longrightarrow \lie b \longrightarrow \lie h \longrightarrow \bc_{\lambda_1 + \lambda_2}.
\] 
We consider the induced module for the subalgebra $(\lie b \otimes 1) \oplus (\lie {sl}_n \otimes t\bc[t]) \subset \lie{sl}_n \otimes \bc[t]$
\[
\Ind_{\lie b \otimes \bc[t]}^{\lie b \otimes 1 \oplus \lie {sl}_n \otimes t\bc[t]} \bc_{\lambda_1 + \lambda_2}
\]
and denote by $M_{\lambda_1, \lambda_2}$ the quotient by the left ideal generated by
\[
\lie{n}^- \otimes t^2\bc[t] \text{ and }  (f_{\alpha} \otimes t)^{\min \{\lambda_1(h_\alpha), \lambda_2(h_\alpha) \} + 1 }, \, \forall \; \alpha \in R^+.
\]
We introduce the $\lie{sl}_n \otimes \bc[t]$-module $F_{\lambda_1, \lambda_2}$ as the maximal integrable (as a $\lie{sl}_n$-module) quotient 
\[
F_{\lambda_1, \lambda_2} = \overline{\Ind_{\lie b \otimes 1 \oplus \lie{sl}_n \otimes t\bc[t]}^{\lie{sl}_n \otimes \bc[t]} M_{\lambda_1, \lambda_2}}.
\]

Due to the construction, we can give defining relations on a generator of $F_{\lambda_1, \lambda_2}$.
\begin{prop}\label{relation}
Let $\lambda_1, \lambda_2 \in P^+$ and $\lambda = \lambda_1 +  \lambda_2$. Then $F_{\lambda_1, \lambda_2}$ is the $\lie{sl}_n \otimes \bc[t]$-module generated through $w$ with relations
\[ \lie{n}^+ \otimes \bc[t].w = 0 , \ \lie{h} \otimes t \bc[t].w = 0 , \ \lie{n}^- \otimes t^2\bc[t] = 0 \]
and for all $\alpha \in R^+$ and $h \in \lie h$:
\[
0 = (f_{\alpha} \otimes 1)^{\lambda(h_\alpha) + 1}.w = (f_\alpha \otimes t)^{\min\{\lambda_1(h_\alpha), \lambda_2(h_\alpha) \}+1} = (h\otimes 1 - \lambda(h)).w.
\]
\end{prop}
\begin{proof}
We have to deal with the $\lie{sl}_n$-relation only. But since $F_{\lambda_1, \lambda_2}$ is integrable we have immediately $(f_\alpha \otimes 1)^{\lambda(h_\alpha) + 1}.\mathbbm{1} = 0$. Therefore $F_{\lambda_1, \lambda_2}$ is a quotient of the module given by the relations in the proposition. On the other hand, every module satisfying the relations is an integrable quotient of $\Ind_{\lie b \otimes 1 \oplus \lie{sl}_n \otimes t\bc[t]}^{\lie{sl}_n \otimes \bc[t]} M_{\lambda_1, \lambda_2}$.
\end{proof}

\begin{prop}\label{firstprop}
Let $\lambda_1, \lambda_2 \in P^+$. Then
\begin{enumerate}
\item $F_{\lambda_1, \lambda_2}$ is a non-negatively graded $\lie {sl}_{n} \otimes \bc[t]$-module.
\item $F_{\lambda_1, \lambda_2}$ is finite-dimensional.
\item $F_{\lambda_1, \lambda_2}= \bigoplus_{s \geq 0} F_{\lambda_1, \lambda_2}^s$, and $ F_{\lambda_1, \lambda_2}^s$ is a $\lie{sl}_n$-module.

\item $F_{\lambda_1, \lambda_2}$ has a unique simple quotient isomorphic to $V(\lambda_1 + \lambda_2)_0$.
\item $\lie{sl}_{n} \otimes t^2 \bc[t]. F_{\lambda_1, \lambda_2} = 0$, and hence $F_{\lambda_1, \lambda_2}$ is a $\lie{sl}_n \otimes \bc[t]/(t^2)$-module.
\end{enumerate}
\end{prop}
\begin{proof}
Part $(1)$ is clear, since the defining relations of  $F_{\lambda_1, \lambda_2}$ are homogeneous and $U(\lie{sl}_n \otimes \bc[t])$ is non-negatively graded.
 Due to the defining relations, $F_{\lambda_1, \lambda_2}$ is a quotient of the local graded Weyl module for $U(\lie{sl}_n \otimes \bc[t])$ of highest weight $\lambda_1 + \lambda_2$, $W_{\bc[t]}(0, \lambda_1 + \lambda_2)$ (see Proposition~\ref{f-weyl} or \cite{CP01} for details, they are not relevant here). In \cite{CP01} it is shown that this local graded Weyl module is finite-dimensional, which implies $(2)$. \\
Now, as $\lie{sl}_n \cong \lie{sl}_n \otimes 1 \hookrightarrow \lie{sl}_n \otimes \bc[t]$, $F_{\lambda_1, \lambda_2}$ is also a finite-dimensional $\lie{sl}_n$-module, hence decomposes into a direct sum of simple finite-dimensional $\lie{sl}_n$-modules. Moreover, as $U(\lie{sl}_n)$ is the degree $0$ part of $U(\lie{sl}_n \otimes \bc[t])$, we see that each graded component $F^s_{\lambda_1, \lambda_2}$ is a $\lie{sl}_n$-module and each simple $\lie{sl}_n$-module is contained in a unique $F_{\lambda_1, \lambda_2}^s$. This implies $(3)$.\\
The degree $0$ component of $F_{\lambda_1, \lambda_2}$ is obviously isomorphic to $V(\lambda_1 + \lambda_2)$ as a $\lie{sl}_n$-module. 
A standard argument shows that
\[
U(\lie{sl}_n \otimes \bc[t])\lie{sl}_n \otimes t\bc[t].\mathbbm{1}
\] 
is the maximal proper submodule not containing $\mathbbm{1}$. The quotient by this submodule is isomorphic to the graded evaluation module $V(\lambda_1 + \lambda_2)_0$, this gives $(4)$. Part $(5)$ follows again immediately from the defining relations.
\end{proof}

Since $F_{\lambda_1, \lambda_2}^s$ is a $\lie {sl}_{n}$-module,  it has a decomposition into a direct sum of simple $\lie{sl}_{n}$-modules
\[ F_{\lambda_1, \lambda_2}^s \cong_{\lie{sl}_{n}} \bigoplus_{\tau \in P^+} V(\tau)^{\oplus a_{\lambda_1, \lambda_2}^\tau(s)}, \text{ for some } a_{\lambda_1, \lambda_2}^{\tau}(s) \geq 0.\]
We set
\[
 a_{\lambda_1, \lambda_2}^{\tau} :=  \sum\limits_{s \geq 0} a_{\lambda_1, \lambda_2}^{\tau}(s),
\]
so we have
\[
\dim \Hom_{\lie{sl}_{n}}( F_{\lambda_1, \lambda_2}, V(\tau)) =  a_{\lambda_1, \lambda_2}^{\tau} .
\]

We see immediately from Proposition~\ref{relation}:
\begin{coro}\label{cor-hw} Let $\lambda_1, \lambda_2 \in P^+$, then
\[
a_{\lambda_1, \lambda_2}^{\lambda_1 + \lambda_2} = 1 \, ; \, a_{\lambda_1, \lambda_2}^{\tau} = 0 \text { for } \tau \nleq \lambda_1 + \lambda_2.
\]
\end{coro}

The main theorem of the paper is the following:
\begin{thm}\label{main-thm} 
Let $\lambda_1, \lambda_2\in P^+$, then we have $ a_{\lambda_1, \lambda_2}^{\tau} \geq c_{\lambda_1, \lambda_2}^{\tau}$, $\forall \; \tau \in P^+$.\\ Moreover:
\begin{enumerate}
\item (Pieri rules) Let $\lambda_1 \in P^+$, $\lambda_2 \in \{ \omega_j, k \omega_1\}$ for some $j \in I$ or $k  \geq 1$, then:  $a_{\lambda_1, \lambda_2}^{\tau} = c_{\lambda_1, \lambda_2}^{\tau}$, $\, \forall \, \tau \in P^+$.
\item  Let $\lambda_1 = m_i \omega_i, \lambda = m_j \omega_j$ for some $i, j \in I, m_i, m_j \geq 0$, then: $a_{\lambda_1, \lambda_2}^{\tau} = c_{\lambda_1, \lambda_2}^{\tau}$, $\, \forall \, \tau \in P^+$.
\item If $\lambda_1 \gg \lambda_2$, then: $a_{\lambda_1, \lambda_2}^{\tau} = c_{\lambda_1, \lambda_2}^{\tau}$, $\, \forall \, \tau \in P^+$.
\end{enumerate}
\end{thm}
The proofs will be given in the following sections, but we should note the following here:
\begin{rem}\label{first-rem} In the proof we will see that $\lambda_1 \gg \lambda_2$ can be made precise, by requesting
 \[
c_{\lambda_1, \lambda_2}^{\tau} = \dim V(\lambda_2)_{\tau - \lambda_1}
\]
for all $\tau \in P^+$. Note that this is equivalent to $\lambda_1 + w(\lambda_2) \in P^+$ for all $w \in W$, the Weyl group of $\lie{sl}_n$.
\end{rem}
\begin{rem}
From the work \cite{CV13, V13} one can deduce further that $a_{\lambda_1, \lambda_2}^{\tau} = c_{\lambda_1, \lambda_2}^{\tau}$ if $\lambda_1 = m \omega_i$ and $\lambda_2(h_\theta) \leq m$ (where $\theta$ is the highest root of $\lie{sl}_n$). The authors were using relations on Demazure modules and their fusion products, generalizing an approach presented in \cite{FoL06}. This of course includes $(2)$ of the theorem but we give a new proof here that might be generalized to other but rectangular weights.
\end{rem}


\section{PBW filtration and polytopes}\label{PBW}
In this section we recall the PBW filtration and we will see how the results from \cite{FFoL11a} can be adapted here in order to understand the $\lie{sl}_{n}$-structure on $F_{\lambda_1, \lambda_2}$.\\
By the PBW theorem and the construction of $F_{\lambda_1, \lambda_2}$ as an induced module we know that
\[ F_{\lambda_1, \lambda_2} = U(\lie{n^-})U(\lie{n^-} \otimes t). \mathbbm{1}.\]
In order to understand the $\lie{sl}_{n}$-decomposition of $ F_{\lambda_1, \lambda_2}$ it would be sufficient to parametrize all $\lie{sl}_{n}$-highest weight vectors.
The equation above suggests, that this set of highest weight vectors should be controlled by $U(\lie{n^-} \otimes t). \mathbbm{1}$. We start with analyzing this.

\subsection{}
We recall the notion of Dyck path from \cite{FFoL11a}: \\
A Dyck path of length $s$ is a sequence of positive roots
\[
\mathbf{p} = (\beta_1, \ldots, \beta_s)
\]
with $s\geq 1$ and such that if $\beta_i = \alpha_{k,\ell}$ then $\beta_{i+1} \in \{ \alpha_{k+1, \ell}, \alpha_{k, \ell +1} \}$. If $\beta_1 = \alpha_{k_1, \ell_1}$ and $\beta_s = \alpha_{k_s, \ell_s}$, then we call 
\[
\alpha_{k_1, \ell_s} \text{ the base root of the path } \mathbf{p}, \text{ denoted by } \beta(\mathbf{p})
\]
Denote the set of all Dyck paths by $\mathbb{D}$.

\subsection{}
The PBW filtration on $U(\lie{n}^-)$ is given as follows:
\[
U(\lie{n}^-)^{\leq s} = \left\langle x_1 \cdots x_r \, | \, 0 \leq r\leq s \text{ and } x_i \in \lie{n}^-  \right\rangle_{\bc}
\]
The associated graded algebra is a commutative algebra isomorphic to $S(\lie n^-)$, the polynomial ring in $\lie{n}^-$. The adjoint action of $\lie n^+$ on $\lie{sl}_n$ induces an action $\circ$ on $S(\lie n^-)$.\\
We fix a tuple of non-negative integers
\[
\mathbf{a} := (a_\alpha) \in \bz_{\geq 0}^{n(n-1)/2}
\]
and consider the ideal $\mathcal{I}(\mathbf{a}) \subset S(\lie n^-)$ given by
\[
\mathcal{I}(\mathbf{a}) = S(\lie n^-)\left\langle \sum_{\alpha \in R^+}  U(\lie n^+) \circ f_\alpha^{a_{\alpha}+1} \right\rangle.
\]

\subsection{}
We fix $\mathbf{a} = (a_\alpha)$ and define a polytop in $ \mathbb{R}^{n(n-1)/2}$:
\[ \mathcal{P}(\mathbf{a}) = \left\{ (x_\alpha) \in \mathbb{R}^{n(n-1)/2} \, | \, \forall \; \mathbf{p} \in \mathbb{D}: \sum_{ \alpha \in \mathbf{p}} x_\alpha \leq a_{\beta(\mathbf{p})}\right\} .\]
We denote 
\[
S(\mathbf{a}) = \mathcal{P}(\mathbf{a}) \cap  \mathbb{Z}_{\geq 0}^{n(n-1)/2}
\]
the set of integer points in $ \mathcal{P}(\mathbf{a})$.\\
This construction of the polytope covers the cases considered in \cite{FFoL11a, FFoL11b, FFoL13, FFoL13a, G11, BD14}, where $a_\alpha := \lambda(h_\alpha)$ for some fixed $\lambda \in P^+$.\\
\noindent
We define further the degree and the weight of an integer point: Let $\mathbf{s} = (s_\alpha) \in  \mathbb{Z}_{\geq 0}^{n(n-1)/2}$, then
\[ \deg (\mathbf{s}) = \sum_{\alpha \in R^+} s_\alpha \text{ and }  \wt (\mathbf{s}) = \sum_{\alpha \in R^+} s_\alpha \alpha \in P. \]

\subsection{}
Although, our approach generalizes the construction provided in \cite{FFoL11a}, we obtain  a similar result on a spanning set of $S(\lie n^-)/ \mathcal{I}(\mathbf{a})$ (see \cite[Theorem 2]{FFoL11a}). For this denote
\[
\mathbf{f}^{\mathbf{t}} = \prod_{\alpha \in R^+} f_\alpha^{t_\alpha} \in S(\lie{n}^-) \text{ where } \mathbf{t} = (t_\alpha) \in  \mathbb{Z}_{\geq 0}^{n(n-1)/2}.
\]
\begin{lem}\label{span-set}
We fix  $\mathbf{a} = (a_\alpha) \in  \mathbb{Z}_{\geq 0}^{n(n-1)/2}$, then
\[
\{ \overline{\mathbf{f}^{\mathbf{s}}} \; | \;\mathbf{s} \in S(\mathbf{a}) \}
\]
is a spanning set of $S(\lie n^-)/ \mathcal{I}(\mathbf{a})$.
\end{lem}
\begin{proof}Here we follow the idea in \cite{FFoL11a}. $\lie{n}^+$ acts by differential operators on $S(\lie{n}^-)$, namely $e_\alpha \circ f_\beta = f_{\beta - \alpha}$ or $0$ if $\beta - \alpha$ is  a positive root. Using these differential operators and an appropriate total order $\prec$ on the monomials in $S(\lie{n}^-)$, we can prove in exactly the same way as \cite[Proposition 1]{FFoL11a} a straightening law. Namely if $\mathbf{s} \notin S(\mathbf{a})$, then
\[
\overline{\mathbf{f}^{\mathbf{s}}} = \overline{\sum_ { \mathbf{t} \prec \mathbf{s} } c_{\mathbf{t}} \mathbf{f}^{\mathbf{t}}}.
\]
This implies now the lemma. For more details we refer to \cite{FFoL11a}.
\end{proof}

In \cite{FFoL11a}, and the case $a_\alpha := \lambda(h_\alpha)$, for some fixed $\lambda \in P^+$, it was further proved that this set is in fact a basis. We can not prove this here and although we conjecture that this is also true in our generality.

\subsection{}
By construction  $M_{\lambda_1, \lambda_2}$ is a cyclic $U(\lie{n^-} \otimes t)$-module. So there exists an ideal $\mathcal{I}_{\lambda_1, \lambda_2}$ such that
\[ M_{\lambda_1, \lambda_2} \cong U(\lie{n^-} \otimes t)/\mathcal{I}_{\lambda_1, \lambda_2}.\]
Since $M_{\lambda_1, \lambda_2}$ is a $\lie{n}^+$-module, the ideal  $\mathcal{I}_{\lambda_1, \lambda_2}$ is stable under the adjoint action of $\lie{n}^+$ (on $U(\lie{n}^- \otimes t)$). 
Moreover the action is a graded action (where $\lie{n}^+$ has degree $0$). Note that we have the identification
\[
S(\lie n^-) \cong U(\lie{n}^-  \otimes t) \subset U(\lie{sl}_n \otimes \bc[t]/(t^2))  \text{ via } \mathbf{f}^{\mathbf{t}} \mapsto \prod_{\alpha} (f_\alpha \otimes t)^{t_\alpha}.
\]
Then we have the obvious proposition:
\begin{prop}\label{prop-maps} For $\lambda_1, \lambda_2 \in P^+$ we set  
\[
\mathbf{a} := (a_\alpha) \text{ where } a_\alpha = \min\{\lambda_1(h_\alpha), \lambda_2(h_\alpha)\}.
\] 
Then we have maps of $S(\lie n^-)$-modules
\[ S(\lie n^-) / \mathcal{I}(\mathbf{a}) \twoheadrightarrow M_{\lambda_1, \lambda_2} \twoheadrightarrow U(\lie n^- \otimes t).\mathbbm{1} \subset F_{\lambda_1, \lambda_2}.\]
\end{prop}

To emphasize the dependence on $\lambda_1, \lambda_2$, we denote the set of integer points $S(\mathbf{a})$ in this case by $S(\lambda_1, \lambda_2)$. Then, combining Proposition~\ref{prop-maps} and Lemma~\ref{span-set}, we have:
\begin{coro}\label{span-m}
Fix $\lambda_1, \lambda_2 \in P^+$, then 
\[
\{ \mathbf{f}^{\mathbf{s}}.\mathbbm{1} \; | \; \mathbf{s} \in S(\lambda_1, \lambda_2) \}
\]
 is, via the identification, a spanning set for $M_{\lambda_1, \lambda_2}$ and hence for $U(\lie n^- \otimes t).\mathbbm{1} \subset F_{\lambda_1, \lambda_2}$.
\end{coro}

\subsection{}
In order to identify the $\lie{sl}_n$-highest weight vectors in $F_{\lambda_1, \lambda_2}$ with images of  $\prod_{\alpha} (f \otimes t)^{s_\alpha} .\mathbbm{1}$ for some $\mathbf{s} \in S(\lambda_1, \lambda_2)$, we introduce an appropriate filtration of $U(\lie n^- \otimes t)$. First we filter by the degree of $t$ and the further by the height of the weights. Finally, we filter further by a total order on the monomials.\\

Recall that $U(\lie n^- \otimes t) \cong S(\lie{n}^-)$ if considered as the subalgebra in $U(\lie{sl}_n \otimes \bc[t]/t^2)$ as we continue to do. Therefore  $U(\lie n^- \otimes t)$ is naturally graded by $t$ and we keep denoting the graded components  $U(\lie n^- \otimes t)^s$. For $\tau \in P$, we denote
\[
U(\lie n^- \otimes t)_\tau = \{ v  \in  U(\lie n^- \otimes t) \ | \ \wt(v) = \tau \}.
\]
All weights of $U(\lie{n^-} \otimes t)$ are in $\bigoplus_{i  \in I} \bz_{\leq 0} \alpha_i$. Let $\tau = \sum_{i \in I} a_ i \alpha_I  \in \bigoplus_{i  \in I} \bz_{\leq 0} \alpha_i$. Then we denote the height of $\tau$
\[
\hei(\tau) := \sum_{i \in I} - a_i.
\] 
So we have a filtration of the graded components
\[
 U(\lie n^- \otimes t)^{s, \leq \ell} = \langle  u \in  U(\lie n^- \otimes t)^s_\tau \ | \ \hei(\tau) \leq \ell \}.
\]
This is spanned by monomials of total degree $s$ and whose weights have height less or equals to $\ell$. \\
On the other hand, $U(\lie n^- \otimes t)$ is $\bz_{\geq 0}^{n(n-1)/2}$ graded. Each graded component is one-dimensional, spanned by $\prod_{\alpha \in R^+} (f_\alpha \otimes t)^{s_\alpha}$ for some $\mathbf{s} \in \bz_{\geq 0}^{n(n-1)/2}$. We order the $n(n-1)/2$-tuples by first ordering the positive roots
\[ \alpha_{i,j} \leq \alpha_{k,\ell} :\Leftrightarrow i <  k \text{ or }  i = k \text{ and } j \leq \ell .\]
Using the lexicographic order $\leq$ we obtain an order on the monomials spanning $U(\lie n^- \otimes t)$.\\
Combining this we introduce a finer filtration on $U(\lie n^- \otimes t)^s$. So given $s, \ell \geq 0$ and $\mathbf{n}\in \bz_{\geq 0}^{n(n-1)/2}$ with $\deg (\mathbf{n}) = s$, $\hei(-\wt(\mathbf{n})) = \ell$, we have
\[
 U(\lie n^- \otimes t)^{s, \leq \ell}_{\leq  \mathbf{n}} =  U(\lie n^- \otimes t)^{s, < \ell} + \left\langle \prod_{\alpha \in R^+} (f_\alpha \otimes t)^{m_\alpha} \ | \ \deg(\mathbf{m}) = s \, , \, \hei( -\wt \mathbf{m}) = \ell \, , \, \mathbf{m} < \mathbf{n} \right\rangle_{\bc}.
\]

\subsection{}
We turn back to the module $F_{\lambda_1, \lambda_2}$ and recall its graded components $F^s_{\lambda_1, \lambda_2}$. We define
\[
\mathcal{F}^{\leq \ell}(F_{\lambda_1, \lambda_2}^s) := U(\lie{sl}_n) U(\lie n^- \otimes t)^{s, \leq \ell} .\mathbbm{1} \subseteq F_{\lambda_1, \lambda_2}^{s}.
\]
By construction
\[
\mathcal{F}^{\leq \ell}(F_{\lambda_1, \lambda_2}^s) /\mathcal{F}^{ < \ell}(F_{\lambda_1, \lambda_2}^s) 
\]
is a $\lie{sl}_n$-module and we have the following 
\begin{lem}
Let $\mathbf{s} \in S(\lambda_1, \lambda_2)$, then the image of 
\[
\mathbf{f}^{\mathbf{s}}.\mathbbm{1} \in \mathcal{F}^{\leq \ell}(F_{\lambda_1, \lambda_2}^s) /\mathcal{F}^{< \ell}(F_{\lambda_1, \lambda_2}^s)
\]
 is either $0$ or a $\lie{sl}_n$-highest weight vector of weight $\lambda_1 + \lambda_2 - \wt(\mathbf{s} )$.
\end{lem}
\begin{proof} Since $\hei(e_\beta) > 0 $, we see using the commutator relations that
\[
e_{\beta} \prod_{\alpha}(f_\alpha \otimes t)^{s_{\alpha}} \in U(\lie n^- \otimes t)^{s, \leq \ell} U(\lie{n}^+)_+ + U(\lie{n}^- \otimes t)^{s, < \ell} U(\lie{n}^+).
\]
This implies that 
\[
e_{\alpha} \prod_{\alpha}(f_\alpha \otimes t)^{s_{\alpha}} = 0 \in \mathcal{F}^{\leq \ell}(F_{\lambda_1, \lambda_2}^s) /\mathcal{F}^{< \ell}(F_{\lambda_1, \lambda_2}^s).
\] 
\end{proof}

We see that, by choosing this appropriate filtration, the highest weight vectors (for the $\lie{sl}_n$-action) of the associated graded module $F_{\lambda_1, \lambda_2}$, are of the form $\mathbf{f}^{\mathbf{s}}.\mathbbm{1}$ for some $\mathbf{s}$. \\
By using the refinement of the filtration we can say even more. So given $s, \ell \geq 0$ and $\mathbf{n}\in \bz_{\geq 0}^{n(n-1)/2}$ with $\deg (\mathbf{n}) = s$, $\hei(-\wt(\mathbf{n})) = \ell$, we have
\[
\mathcal{F}_{\leq \mathbf{n}}^{\leq \ell}(F_{\lambda_1, \lambda_2}^s) := U(\lie{sl}_n) U(\lie n^- \otimes t)_{\leq \mathbf{n}}^{s, \leq \ell} .\mathbbm{1} \subset F_{\lambda_1, \lambda_2}^{s}.
\]
Then the graded components 
\[
\mathfrak{G}_{\mathbf{n}}^{s, \ell}(F_{\lambda_1, \lambda_2})  := \mathcal{F}_{\leq \mathbf{n}}^{\leq \ell}(F_{\lambda_1, \lambda_2}^s) / \left( \mathcal{F}^{<  \ell}(F_{\lambda_1, \lambda_2}^s) + \sum_{ \mathbf{m} < \mathbf{n} } \mathcal{F}_{\leq \mathbf{m}}^{\leq  \ell}(F_{\lambda_1, \lambda_2}^s) \right)
\]
are simple $\lie{sl}_n$-modules.

\subsection{}
We have seen in Corollary~\ref{span-m} that the monomials corresponding to points in $S(\lambda_1, \lambda_2)$ are a spanning set of $U(\lie n^{-} \otimes t). \mathbbm{1}$. 
\begin{defn}
We say $\mathbf{n} \in S(\lambda_1, \lambda_2)$ is a \textit{ highest weight point } if $\mathfrak{G}_{ \mathbf{n}}^{s, \ell}(F_{\lambda_1, \lambda_2}) $ is non-zero for $s = \deg(\mathbf{n})$ and $\ell = \hei(-\wt(\mathbf{n}))$. The set of highest weight points is denoted $S_{hw}(\lambda_1, \lambda_2)$.
\end{defn}

Note that, since $F_{\lambda_1, \lambda_2}$ is an integrable $\lie{sl}_n$-module, we have for all $\mathbf{s} \in S_{hw}(\lambda_1, \lambda_2)$
\[
\lambda_1  + \lambda _ 2 - \wt(\mathbf{s}) \in P^+.
\]

\begin{coro}\label{cor-upper} For $\lambda_1, \lambda_2, \tau \in P^+$ we have
\[
\dim \Hom_{\lie{sl}_n} (F_{\lambda_1, \lambda_2}, V(\tau)) = \sharp \{  \mathbf{s} \in S_{hw}(\lambda_1, \lambda_2) \, | \, \wt(\mathbf{s}) = \lambda_1 + \lambda_2 - \tau \}
\]
Moreover
\[
\dim \Hom_{\lie{sl}_n} (F^s_{\lambda_1, \lambda_2}, V(\tau)) = \sharp \{  \mathbf{s} \in S_{hw}(\lambda_1, \lambda_2) \, | \, \wt(\mathbf{s}) = \lambda_1 + \lambda_2 - \tau \text{ and } \deg( \mathbf{s}) = s.\}
\]
\end{coro}


\section{Fusion products}\label{fusion}
In this section we recall the fusion product of two simple $\lie{sl}_{n}$-modules and work out the relation to the modules $F_{\lambda_1, \lambda_2}$.

\subsection{}
The following construction is due to \cite{FL99}. Recall the grading on $U(\lie{sl}_{n} \otimes \bc[t])$  given by the degree function on $\bc[t]$
\[
 U(\lie{sl}_n \otimes \bc[t]) ^r = \{ u \in U(\lie{sl}_n \otimes \bc[t]) \; | \; \deg(u) \leq r \}.
\]
Then $\mathcal{F}^{0} = U(\lie {sl}_n)$ and we set $\mathcal{F}^{-1} = 0$.\\
Let $V(\lambda_1), \ldots, V(\lambda_k)$ be simple $\lie{sl}_n$-modules of highest weights $\lambda_1, \ldots, \lambda_k$. Further let $c_1, \ldots, c_k$ be pairwise distinct complex numbers. Then $V(\lambda_i)$ can be endowed with the structure of a $\lie{sl}_n \otimes \bc[t]$-module via
\[
x \otimes p(t).v = p(c_i)x.v \; \; \text{ for all } x \in \lie{sl}_n, p(t) \in \bc[t], v \in V(\lambda_i),
\]
we denote this module $V(\lambda_i)_{c_i}$.
Then 
\[
V(\lambda_1)_{c_1} \otimes \cdots \otimes V(\lambda_k)_{c_k}
\]
is cyclic generated by the tensor product of highest weight vectors  $v_{\lambda_1} \otimes \cdots \otimes v_{\lambda_k}$ (even more it is simple \cite{Rao93, CFK10}). The grading on $U(\lie{sl}_n \otimes \bc[t])$ induces a filtration on $V(\lambda_1)_{c_1} \otimes \cdots \otimes V(\lambda_k)_{c_k}$
\[
U(\lie{sl}_n \otimes \bc[t])^{\leq r}. v_{\lambda_1} \otimes \cdots \otimes v_{\lambda_k}.
\]
Since $U(\lie{sl}_n \otimes \bc[t])$ is graded, the associated graded is again a module for $U(\lie{sl}_n \otimes \bc[t])$, denoted usually by
\[
V(\lambda_1)_{c_1} \ast \cdots \ast V(\lambda_k)_{c_k},
\]
and is called the \textit{fusion product}.  Recall that the graded components are $U(\lie {sl}_n)$-modules, since $U(\lie {sl} \otimes 1)$ is the degree $0$ component of $U(\lie{sl}_n \otimes \bc[t])$. Further, since we have not changed the $\lie{sl}_n$-structure in this construction:

\begin{coro}\label{fusion-decom}
Let $\lambda_1, \lambda_2 \in P^+$, $c_1 \neq c_2 \in \bc$, then for all $\tau \in P^+$
\[
\dim \Hom_{\lie{sl}_n} (V(\lambda_1)_{c_1} \ast V(\lambda_2)_{c_2}, V(\tau)) = c_{\lambda_1, \lambda_2}^{\tau}.
\]
\end{coro}

\subsection{}

\begin{lem}\label{whatever}
For $\lambda_1, \lambda_2 \in P^+$, $c_1 \neq c_2 \in \bc$ we have a surjective map of $\lie{sl}_{n} \otimes \bc[t]$-modules:
\[ F_{\lambda_1, \lambda_2} \twoheadrightarrow V(\lambda_1)_{c_1} \ast V(\lambda_2)_{c_2},\]
moreover $a_{\lambda_1, \lambda_2}^{\tau} \geq c_{\lambda_1, \lambda_2}^{\tau}$, $ \forall \, \tau \in P^+$.
\end{lem}
\begin{proof}
We prove the $\lie{sl}_2$-case first. Here dominant integral weights are parameterized by $\bz_{\geq 0}$, and for $k \geq 0$ let $V(k) =Sym^k \bc^2$. Fix $k \geq m \geq 0$. Then 
\[ 
\dim (V(k)_{c_1} \otimes V(m)_{c_2})_{k + m - 2 \ell} = \begin{cases} \ell + 1 \text{ for } 0 \leq \ell \leq m \\ m + 1 \text{ for } m \leq \ell \leq  k \\  k + m + 1- \ell \text{ for } k \leq \ell \leq k + m \end{cases}
\]
Since $c_1 \neq c_2$, we se, using the Vandermonde determinant, that 
\[ (f_\alpha \otimes t)^{m+1} v_k \otimes v_m \in \langle (f_\alpha \otimes 1)^{m+1} v_k \otimes v_m, \ldots ,(f_\alpha \otimes 1) (f_\alpha \otimes t)^m v_k \otimes v_m \rangle_{\bc}\]
since the $k-2$-weight space is at most $m+1$-dimensional. This implies that $(f_\alpha \otimes t)^{m+1} v_k \otimes v_m$ is  $0$ in the associated graded module.\\ 
We see further, that the weight space of weight $k+m -2$ is two dimensional and spanned by the vectors $ (f_\alpha \otimes 1) v_k \otimes v_m, (f_\alpha \otimes t) v_k \otimes v_m$. This implies that for $\ell \geq 2$, $(f_\alpha \otimes t^\ell) v_k \otimes v_m = 0$ in the fusion product, similarly we see that for all $\ell \geq 1$, $h \otimes t^\ell v_k \otimes v_m = 0$ in the fusion product.\\
This implies that there is a surjective map of $\lie{sl}_2$-modules
\[ F_{k\omega,m\omega} \twoheadrightarrow V(k)_{c_1} \ast V(m)_{c_2}. \]
Let us turn to the general case. Let $\lambda_1, \lambda_2 \in P^+$, $c_1 \neq c_2 \in \bc$, $\alpha \in R^+$, and let $m = \min \{ \lambda_1(h_\alpha), \lambda_2(h_\alpha) \}$.  By considering the $\lie{sl}_2$-triple $\{ e_\alpha, h_\alpha, f_\alpha\}$ we see with the same argument as  above that
\[
 (f_\alpha \otimes t)^{m+1} v_{\lambda_1} \otimes v_{\lambda_2} \in \operatorname{ span } \{  (f_\alpha \otimes 1)^{m+1} v_{\lambda_1} \otimes v_{\lambda_2}, \ldots , (f_\alpha \otimes 1)(f_\alpha \otimes t)^m v_{\lambda_1} \otimes v_{\lambda_2} \} 
\]
This implies that $ (f_\alpha \otimes t)^{m+1} v_{\lambda_1} \otimes v_{\lambda_2} = 0$ in the associated graded. The remaining defining relations for $F_{\lambda_1, \lambda_2}$ are easily verified.
\end{proof}

Using this lemma we have the following very interesting consequence:
\begin{coro}
If  $\forall \, \tau \in P^+$:  $a^{\tau}_{\lambda_1, \lambda_2} = c^{\tau}_{\lambda_1, \lambda_2}$, then for all $c_1 \neq c_2 \in \bc:$
\[  V(\lambda_1)_{c_1} \ast V(\lambda_2)_{c_2} \cong_{\lie{sl}_{n} \otimes \bc[t] } F_{\lambda_1, \lambda_2}.\]
Moreover, the fusion product in this case is independent of the parameter $c_1, c_2$, providing another proof of a conjecture by B.Feigin and S.Loktev (\cite{FL99}).
\end{coro}
\begin{proof}
By Lemma~\ref{whatever} we have for all $\lambda_1, \lambda_2 \in P^+$ and $c_1 \neq c_2 \in \bc$ a surjective map of $\lie{sl}_n \otimes \bc[t]$-modules
\[
F_{\lambda_1, \lambda_2} \twoheadrightarrow V(\lambda_1)_{c_1} \ast V(\lambda_2)_{c_2}.
\]
With Corollary~\ref{fusion-decom} we know that the multiplicity of $V(\tau)$ in the fusion product is $c_{\lambda_1, \lambda_2}^\tau$. By assumption, this is equal to $a_{\lambda_1, \lambda_2}^\tau$, which is the multiplicity of $V(\tau)$ in $F_{\lambda_1, \lambda_2}$. So the modules are isomorphic as $\lie{sl}_n$-modules and hence by a dimension argument also as $\lie{sl}_n \otimes \bc[t]$-modules.\\
Since $F_{\lambda_1, \lambda_2}$ is a graded module and independent of any evaluation parameter, the same is true for the fusion product $V(\lambda_1)_{c_1} \ast V(\lambda_2)_{c_2}$.
\end{proof}

\section{First proofs for parts of the main theorem}\label{firstproof}
We prove here the $\lie{sl}_2$-case, namely $a_{m \omega_1, k \omega_1}^{\tau} = c_{m \omega_1, k \omega_1}^{\tau}$ for all $m, k \geq 0$ and $\tau \in P^+$. In the following section we prove the $\lambda_1 \gg \lambda_2$-case.

\subsection{}
In this section we consider the $\lie{sl}_2$-case. In this case, dominant integral weights are parametrized by non-negative integers.
\begin{lem}\label{fusion-sl2} Let $m_1, m_2 \geq 0$, then for all $c_1 \neq c_2 \in \bc$
\[ 
F_{m_1 \omega_1, m_2 \omega_1} \cong_{\lie{sl}_2 \otimes \bc[t]} V(m_1 \omega_1)_{c_1} \ast V(m_2 \omega_1)_{c_2}.
\]
Moreover, $a_{m_1 \omega_1, m_2 \omega_1}^{k \omega_1} = c_{m_1 \omega_1, m_2 \omega_1}^{k \omega_1}$.
\end{lem}
This proves Theorem~\ref{main-thm}(1) for $A_1$.
\begin{proof}
Let $m_1, m_2 \in \mathbb{Z}_{\geq 0}$, then by Lemma~\ref{whatever} it suffices to prove that $ \forall \; k \geq 0$
\[
a_{m_1 \omega_1, m_2 \omega_1}^{k \omega_1} \leq  c_{m_1 \omega_1, m_2 \omega_1}^{k \omega_1} \;
\]
Suppose $m_1 \geq m_2$. Then the relations of $F_{m_1 \omega_1, m_2 \omega_2}$ can be rewritten as
\[ (h \otimes 1).\mathbbm{1} = (m_1 + m_2 + 1). \mathbbm{1}; \; ; (f \otimes 1)^{m_1 + m_2 + 1}. \mathbbm{1} = 0 \; ; \; (f \otimes t)^{m_2 + 1}. \mathbbm{1}= 0,\]
while $(\lie {n}^- \otimes t^2 \bc[t] \oplus \lie b \otimes t\bc[t] \oplus \lie {n}^+ \otimes 1).  \mathbbm{1} = 0$. By considering $F_{m_1 \omega_1, m_2 \omega_1}$ as an $\lie{sl}_2$-module we see from the relations, that it is generated by 
\[ \{ \mathbbm{1}, (f \otimes t).\mathbbm{1}, (f \otimes t)^{2}.\mathbbm{1}, \ldots, , (f \otimes t)^{m_2}.\mathbbm{1} \} .\]
This implies that $F_{m_1 \omega_1, m_2 \omega_1}$ is multiplicity free and moreover we see that 
\[a_{m_1 \omega_1, m_2 \omega_1}^{k \omega_1} = 1 \Rightarrow k = m_1 + m_2 - 2 \ell \text{ for some } \ell \in \{0, \ldots, m_2 \}.\]
The famous Clebsch-Gordan formula gives for $k = m_1 + m_2 - 2 \ell \text{ for some } \ell \in \{0, \ldots, m_2 \}$
\[c_{m_1 \omega_1, m_2 \omega_1}^{k \omega_1} =1  \text{ and } c_{m_1 \omega_1, m_2 \omega_1}^{k \omega_1} = 0 \text{ else  }.\]
This implies (with Lemma~\ref{whatever})
\[
a_{m_1 \omega_1, m_2 \omega_1}^{k \omega_1} \leq  c_{m_1 \omega_1, m_2 \omega_1}^{k \omega_1} \leq a_{m_1 \omega_1, m_2 \omega_1}^{k \omega_1}.
\]
\end{proof}
Note here, that this elementary result follows also from \cite{FF02} and \cite{CV13}.

\subsection{}
Let $\lambda_1, \lambda_2 \in P^+$. We say
\[
\lambda_1 \gg \lambda_2 \Leftrightarrow \lambda_1 + w(\lambda_2) \in P^+ \Leftrightarrow c_{\lambda_1, \lambda_2}^{\tau} = \dim V(\lambda_2)_{ \tau - \lambda_1},
\, \forall \tau \in P^+.
\] 
This is certainly satisfied if  $\lambda_1(h_\alpha) \gg \lambda(h_\alpha)$ for all $\alpha \in R^+$. \\

Suppose now $\lambda_1 \gg \lambda_2$, then $\min\{ \lambda_1(h_\alpha), \lambda_2(h_\alpha) \} = \lambda_2(h_\alpha)$. Which implies that if we define
\[
\mathbf{a} \in \bz^{n(n-1)/2} \text{ via } a_\alpha = \min \{ \lambda_1(h_\alpha), \lambda_2(h_\alpha) \}
\]
then $a_\alpha = \lambda_2(h_\alpha)$. Let us denote $V(\lambda_2)^{a}$ the associated graded module obtained through the PBW filtration $U(\lie n^-)$ on the highest weight vector $v_{\lambda_2} \in V(\lambda_2)$ (see \cite{FFoL11a} for more details). This is a module for $S(\lie n^-)$, the associated graded algebra of $U(\lie n^-)$. \begin{prop}\label{classicalpbw}
If $\lambda \gg \lambda_2$, then
\[
S(\lie n^-)/\mathcal{I}(\mathbf{a}) \cong V(\lambda_2)^{a}.
\]
\end{prop}
\begin{proof} This is nothing but \cite[Theorem A]{FFoL11a}.
\end{proof}
We are ready to prove:
\begin{thm}
If $\lambda_1 \gg \lambda_2$, then
\[
a_{\lambda_1, \lambda_2}^{\tau} = c_{\lambda_1, \lambda_2}^{\tau}
\]
and 
\[
F_{\lambda_1, \lambda_2} \cong_{\lie{sl}_n \otimes \bc[t]} V(\lambda_1)_{c_1} \ast V(\lambda_2)_{c_2}
\]
for all $c_1 \neq c_2 \in \bc$.
\end{thm}
\begin{proof}
With Corollary~\ref{cor-upper} we see that 
\[
a_{\lambda_1, \lambda_2}^{\tau} \leq  \sharp \{  \mathbf{s} \in S(\lambda_1, \lambda_2) \, | \, \wt(\mathbf{s}) = \lambda_1 + \lambda_2 - \tau \}.
\]
On the other hand, by Lemma~\ref{whatever}, we have
\[
a_{\lambda_1, \lambda_2}^{\tau} \geq c_{\lambda_1, \lambda_2}^{\tau}.
\]
By assumption $\lambda_1 \gg \lambda_2$, which implies (Remark~\ref{first-rem})
\[
 c_{\lambda_1, \lambda_2}^{\tau} = \dim V(\lambda_2)_{\tau - \lambda_1}.
\]
Now \cite[Theorem B]{FFoL11a} gives in this case a parametrization of a basis of $V(\lambda_2)$ in terms of (in our notation) $S(\lambda_1, \lambda_2)$, namely
\[
\dim V(\lambda_2)_{\tau - \lambda_1} =  \sharp \{  \mathbf{s} \in S(\lambda_1, \lambda_2) \, | \, \wt(\mathbf{s}) = \lambda_2 - (\tau -\lambda_1)\}.
\]
Which implies also
\[
a_{\lambda_1, \lambda_2}^{\tau} \leq c_{\lambda_1, \lambda_2}^{\tau},
\]
hence the equality follows.
\end{proof}


\section{Rectangular weights}\label{KR}
In this section we prove generators and relations for the fusion product of two arbitrary Kirillov-Reshetikhin modules. These modules are defined in the context of simple, finite-dimensional modules for the quantum affine algebra. They are indexed by a node $i \in I$, a level $m$ and an evaluation parameter $a \in \bc(q)^*$ and denoted $KR( m \omega_i, a)$. For more on their importance we refer here to the survey \cite{CH10}.\\
 In this paper we consider the non-quantum analog (obtained through the $q\mapsto 1$ limit). In the $\lie{sl}_n$-case, they are isomorphic to evaluation modules $V(m \omega_i)_{c}$ for some $c \in \bc$.\\
We have seen in Lemma~\ref{whatever} that 
\[
F_{m_i \omega_i, m_j \omega_j} \twoheadrightarrow V(m_i \omega_i)_{c_1} \ast V(m_j \omega_j)_{c_2}
\]
for all $c_1 \neq  c_2$. We want to prove that this map is in fact an isomorphism, so we have to show that for all $\tau \in P^+$
\[
a_{m_i \omega_i, m_j \omega_j}^{\tau} = c_{m_i \omega_i, m_j \omega_j}^{\tau} .
\]

\subsection{}
First, we will give formulas for the right hand side.  We refer here to \cite{N2002} where the decomposition of a tensor product was computed by using combinatorics of Young tableaux. A formula for the tensor product of $V(\lambda_1)$ with $V(\omega_1)$ is given explicitly and as well as the induction procedure for $V(\lambda_2)$. In the special case of $\lambda_1 = m_i \omega_i$ and $\lambda_2 = m_j \omega_j$ one can deduce straightforward that for all $\tau \in P^+$:
\[ c_{m_i \omega_i, m_j \omega_j}^{\tau} \in \{ 0,1\}. \]
Moreover
\begin{prop} For $i \leq j$, $c_{m_i \omega_i, m_j \omega_j}^{\tau}=1$ if and only if (setting $\omega_n = \omega_0 = 0$.)
\[
\tau = m_i \omega_i + m_j \omega_j + \sum_{q \geq 0}^{ \min\{ i, j+i, n-j\}} b_q(\omega_{i -q} + \omega_{j + q} -  \omega_i - \omega_j) 
\]
with 
\[
  \sum_{q \geq 0}^{ \min\{ i, j+i, n-j\}} b_q \leq \min\{ m_i, m_j\} \; \; , \; b_q \geq 0
\]

\end{prop}

\subsection{}
Second, we will compute $a_{m_i \omega_i, m_j \omega_j}^{\tau}$. For this we identify again
\[
\mathbf{f}^{\mathbf{s}} \leftrightarrow \prod_{\alpha } (f_{\alpha} \otimes t)^{s_\alpha}.
\]
Recall, from Section~\ref{PBW} (and \cite{FFoL11a}) that $\lie{n}^+$ acts by differential operators on $S(\lie{n}^-)$. Here, we introduce a new class of operators as follows. Let $R^+_{\lambda_1, \lambda_2} = \{ \alpha \in R^+ \, | \, \lambda_1(h_\alpha) = \lambda_2(h_\alpha) = 0$\}. Then $\lie{n}^-_{\lambda_1, \lambda_2} = \langle f_\alpha \, | \, \alpha \in R^+_{\lambda_1, \lambda_2}\rangle $ is a subalgebra. We define for $\alpha \in R^+_{\lambda_1, \lambda_2}, \beta \in R^+$:
\[ f_\alpha \circ f_\beta \otimes t = \begin{cases} f_{\alpha +  \beta} \otimes t  \text{ if } \alpha + \beta \in R^+ \\ 0  \; \; \; \text{ else }  \end{cases}
\]
This is induced by the adjoint action of $\lie{n}^-$ on $\lie{n}^- \otimes t$ (we normalize if necessary here). Moreover
\begin{prop}
This action induces an action of differential operators on $U(\lie{n}^- \otimes t).\mathbbm{1} \subset F_{\lambda_1, \lambda_2}$. 
\end{prop}
\begin{proof}
This follows easily from the fact that $\lie{n}^-_{\lambda_1, \lambda_2}.\mathbbm{1} = 0 \in  F_{\lambda_1, \lambda_2}$.
\end{proof}

In the following we will abbreviate $f_{\alpha_{k, \ell}}$ with $f_{k, \ell}$, $s_{\alpha_{k, \ell}}$ with $s_{k, \ell}$. Denote further $\mathbf{e}_{k,l}$, the basis vector of $\mathbb{R}^{n(n-1)/2}$ having $1$ for $e_{\alpha_{k, \ell}}$ and $0$ elsewhere. So let $\alpha \in R^+_{\lambda_1, \lambda_2}$ and $\gamma = \alpha + \beta \in R^+$, then 
\[
f_\alpha \circ \mathbf{f}^{\mathbf{e}_\beta} = \mathbf{f}^{\mathbf{e}_\gamma}.
\]

\subsection{}
We turn to the case $\lambda_1 = m_i \omega_i, \lambda_2 = m_j \omega_j$. Let $\mathbf{s} \in S(\lambda_1, \lambda_2)$, then $s_{k, \ell} = 0$ for $\ell < j$ or $k > i$. The following is the crucial lemma, which gives an upper bound for the set of highest weight points.
\begin{lem}\label{diagonal}
Let $i\leq j \in I, m_i, m_j \geq 0$, and $p := \min \{i-1, n-1-j \}$, then
\[
U(\lie{n}^- \otimes t).\mathbbm{1} \subset U(\lie{n}^-) \langle (f_{i,j} \otimes t)^{a_0} (f_{i-1, j+1} \otimes t)^{a_1} \cdots (f_{i- p, j+p})^{a_k}. \mathbbm{1} \, | \,  a_q \geq 0, \forall  \,q \rangle.
\]
Moreover we have
\[
S_{hw}(\lambda_1, \lambda_2) \subseteq \{ \mathbf{s} \in S(m_i \omega_i, m_j \omega_j) \, | \, s_{k, \ell} = 0 \text{ if } (k, \ell) \neq (i-q, j+q) \text{ for some } q \}.
\]

\end{lem}
\begin{proof}
We have seen in Corollary~\ref{span-m}, that 
\[ 
\{ \mathbf{f}^{\mathbf{s}}.\mathbbm{1} \, | \, \mathbf{s} \in S(\lambda_1, \lambda_2) \}
\]
 generates $F_{\lambda_1, \lambda_2}$ as a $U(\lie n^{-})$-modules.\\
In our case $\lambda_1 = m_i \omega_i, \lambda_2 =  m_j \omega_j$ and let $\mathbf{s} \in \bz_{\geq 0}^{n(n-1)/2}$ with $s_{p, q} = 0$ for $q < j$ or $p > i$. Let $k, \ell$ be such that $ i-k > \ell - j$, $s_{k, \ell} \neq 0$ and 
$$
\begin{array}{ll}
\text{Condition }(1):&  s_{r, \ell} = 0, \;  \forall \, r=1, \ldots , k-1 \\
\text{Condition }(2): &  s_{r,s} = 0 \text{ if } r < k \text{ and } s< j + i -r, \text{ then } 
\end{array}
$$
So $\mathbf{s}$ is of the form:
\[
\left(
\begin{array}{cccccccccccc}
0 &     \ldots         &      &      0     &     0                &             \ldots                              &      &  &             \ldots         &             0       & s_{i,j}\\
  &           &  &       &                   &            \ldots                               &       &  &       \ldots            &             s_{i-1, j +1}       & s_{i,j+1}\\
\vdots &   \ddots         &     &   \vdots         &        \vdots              &                                &       &           &     \iddots &               & \vdots\\
 0&  &  &  	0	&	  0           & \ldots  &0  &   s_{i+j - \ell+1, \ell -1}  & \ldots  & & s_{i, \ell-1}\\
0 & \ldots & 0 & s_{k, \ell} & s_{k+1, \ell} & \ldots & s_{i+j - \ell, \ell} & &  \ldots&& s_{i, \ell}\\
s_{1, \ell+1} & \ldots &  &  &  & \ldots & & &  \ldots&& s_{i, \ell+1}\\
\vdots & \ddots &  &  & & \ddots& & &\ddots &&\vdots\\
s_{1, n-1} & \ldots &  & & & \ldots &  & &  \ldots&& s_{i, n-1}
\end{array}
\right)
\]
We consider
\[
f_{k,k} \circ \left(   \mathbf{f}^{\mathbf{s}  - \mathbf{e}_{k, \ell} + \mathbf{e}_{k+1,\ell}}\right).\mathbbm{1}
\]
By expanding this we see that
\[
(s_{k+1, \ell}+1)\mathbf{f}^{\mathbf{s}}. \mathbbm{1} = \left[ f_{k,k} \circ \left(   \mathbf{f}^{\mathbf{s} - \mathbf{e}_{{k, \ell}} + \mathbf{e}_{{k+1,\ell}}}\right) - 
\sum\limits_{z = \ell +1}^{n-1} s_{k+1, z}
 \left(  \mathbf{f}^{\mathbf{s} - \mathbf{e}_{k,\ell}  + \mathbf{e}_{k+1, \ell} + \mathbf{e}_{k,  z} - \mathbf{e}_{k+1, z} } \right)\right]. \mathbbm{1}. 
\]
By iterating this we see that 
\[ 
\mathbf{f}^{\mathbf{s}}. \mathbbm{1} \in \sum_{ \mathbf{n}_\alpha} U(\lie{n}^-) \mathbf{f}^{\mathbf{n}}. \mathbbm{1}
\]
where the sum is over all $\mathbf{n}  \in \bz_{\geq 0}^{n(n-1)/2}$ satisfying Condition (1) and (2) and moreover $n_{k,\ell}=0$.\\
Using induction along the first row, then along the second row etc, we see that 
\[ 
\mathbf{f}^{\mathbf{s}}. \mathbbm{1} \in \sum_{ \mathbf{n}} U(\lie{n}^-) \mathbf{f}^{\mathbf{n}}. \mathbbm{1}
\]
where $n_{{k,\ell}} = 0 $ for all $k,\ell$ with $i - k > \ell - j$.\\
A similar computation for the roots below the diagonal shows that we can assume also $n_{{ k , \ell}} = 0$ for all $(k, \ell) \neq (i - q, j + q)$ for some $q$. This proves the first part of the lemma. The claim on highest weight points follows now from the definition of $F_{\lambda_1, \lambda_2}$, namely
\[ (f_{i - q, j+ q} \otimes t)^K. \mathbbm{1} = 0 \text{ for } K \geq \min\{ m_i , m_j \}.\]
\end{proof} 
The following gives a stricter upper bound for the set of highest weight points.
\begin{prop}
Let $i\leq j \in I, m_i, m_j \geq 0$, $p = \min \{i-1, n-1-j \}$, then
\[
S_{hw}(\lambda_1, \lambda_2) \subseteq \{ a_0 \mathbf{e}_{i,j} + a_1 \mathbf{e}_{i-1, j+1} + \ldots + a_p \mathbf{e}_{i-p, j+p} \, | \, \min\{ m_i, m_j \} \geq a_0 \geq a_1 \geq \ldots \geq a_ p \geq 0 \}.
\]
\end{prop}
\begin{proof}
We just have to check that these are the only points of the ones described in Lemma~\ref{diagonal} whose monomials applied on $\mathbbm{1}$ give vectors of dominant weight. For this, the weight of the vector 
\[  (f_{i,j} \otimes t)^{a_0} (f_{i-1, j+1} \otimes t)^{a_1} \cdots (f_{i- p, j+p})^{a_k}. \mathbbm{1}\]
is equal to 
\[
m_i \omega_i + m_j \omega_j + \sum_{q = 0}^{p} a_p (\omega_{i-q} + \omega_{j+q} -  \omega_{i-q-1} - \omega_{j+q+1} ).
\]
This is equal to 
\[
(m_i - a_0) \omega_i + (a_0 - a_1) \omega_{i-1} + \ldots + (a_{p-1} - a_p) \omega_p + (m_j - a_0) \omega_j + \ldots + (a_{p-1} - a_p)\omega_{j+p}
\]
which is dominant if and only if $a_0 \geq a_1 \geq \ldots \geq a_p$.
\end{proof}

Keep the notation from the proof and set $b_i = a_i - a_{i+1} \geq 0$, then the weight of 
\[  (f_{i,j} \otimes t)^{a_0} (f_{i-1, j+1} \otimes t)^{a_1} \cdots (f_{i- p, j+p})^{a_k}. \mathbbm{1}\]
is equal
\[
m_i\omega_i + m_j \omega_j + \sum_{q = 0}^{\min\{ i-1, n-1 -j\}} b_q(\omega_{i-q} + \omega_{j+q} - \omega_{i} - \omega_{j})
\]
with 
\[
 \sum_{q = 0}^{\min\{ i-1, n-1 -j\}} b_q = a_0 \leq \min\{m_i, m_j\}.
\]
This implies
\begin{thm}
Let $i, j\in I$, $m_i, m_j \geq 0$, then
\[
a_{m_i \omega_i, m_j \omega_j}^{\tau} = c_{m_i \omega_i, m_j \omega_j}^{\tau}
\]
for all $\tau \in P^+$ and hence
\[
F_{m_i \omega_i, m_j \omega_j} \cong_{\lie{sl}_n \otimes \bc[t]} V(m_i \omega_i)_{c_1} \ast V(m_j \omega_j)_{c_2}
\]
for all $c_1 \neq c_2 \in \bc$.
\end{thm}


\section{The Pieri rules}\label{pieri}
In this section we want to compute the $\lie{sl}_n$ decomposition on $F_{\lambda, \omega_j}$ and $F_{\lambda, k \omega_1}$. Mainly, we want to identify them with the fusion product of $V(\lambda)$ and $V(\omega_j)$ (resp. $V(k \omega_1)$). As for the Kirillov-Reshetikhin modules we will show that $a_{\lambda, \omega_j}^{\tau} = c_{\lambda, \omega_j}^{\tau}$ for all $\tau$ and similar for $k \omega_1$. Let us start with the latter case.\\
On one hand, using again the Young tableaux combinatorics from \cite{N2002}, we see that the highest weight vectors of $V(\lambda) \otimes V(k \omega_1)$ are parameterized by the set 
\[
T_{\lambda, k \omega_1} := \{ ( b_1, \ldots, b_{n} ) \in \mathbb{Z}^{n}_{\geq 0} \, |, \,  b_1 + \ldots + b_{n} =  k \text{ and } b_j \leq m_{j-1} \; \forall \; j = 2, \ldots, n \}
\]
where $\lambda = \sum m_i \omega_i$. \\
Let $\mathbf{s} \in S_{hw}(\lambda, k\omega_1) \subseteq S(\lambda, k \omega_1)$, then $s_{i,j} = 0$ if $i \neq 1$. So there is no confusion if we write in the following $s_j$ for $s_{1,j}$. We have $s_{1} + \ldots + s_{n-1} \leq k$. Suppose now $s_{j} > m_j$ for some $j > 1$, then by definition of $F_{\lambda, k \omega_1}$
\[
f_{j}^{s_{j}}. \mathbbm{1} = 0
\]
This implies, recall the notation of Section~\ref{KR}, 
\[
\mathbf{f}^{\mathbf{s} + s_j \mathbf{e}_{1, j-1} - s_j \mathbf{e}_{1, j}}f_{j}^{s_{j}}. \mathbbm{1} = 0
\]
Using commutator relations we have for some constants $c_k$
\[
\sum_{k = 0}^{s_j}c_k f_{j}^{k}  \mathbf{f}^{\mathbf{s} + k ( \mathbf{e}_{1, j-1} -  \mathbf{e}_{1, j})}. \mathbbm{1} = 0
\]
This implies that 
\[
\mathbf{f}^{\mathbf{s}}.\mathbbm{1} \in \sum_{\mathbf{n}} U(\lie n^-) \mathbf{f}^{\mathbf{n}}.\mathbbm{1}
\]
for some $\mathbf{n}$ with $n_\ell = s_\ell$ for $\ell > j$ and $ n_j < s_j$. But this is a contradiction to $\mathbf{s}\in S_{hw}(\lambda, k\omega_1)$.\\
This implies that if $\mathbf{s} \in S_{hw}(\lambda, k\omega_1)$ we have 
\[
s_{i,j} = 0 \text{ for } i \neq 1 \; , \; s_{j} \leq m_j \; , \;  s_{1} + \ldots + s_{n-1} \leq k.
\]
This implies $|S_{hw}(\lambda, k\omega_1) | \leq |T_{\lambda, k \omega_1} |$. Using now Lemma~\ref{whatever} we have equality here and so 
\begin{lem}
For $\lambda \in P^+$, $k \geq 0$, we have
\[
a_{\lambda, k \omega_1}^{\tau} = c_{\lambda, k\omega_1}^{\tau} \text{ for all } \tau \in P^+
\]
and so for all $c_1 \neq c_2 \in \bc$
\[
F_{\lambda, k \omega_1} \cong V(\lambda)_{c_1}\ast V(k\omega_1)_{c_2}.
\]
\end{lem}

\subsection{}
We consider here the $\omega_j$-case. As before, using Young Tableaux combinatorics from \cite{N2002}, we have that the highest weight vectors of $V(\lambda) \otimes V(\omega_j)$ are parameterized by the set ($\lambda = \sum m_i \omega_i$)
\[
T_{\lambda, \omega_j} := \{ ( b_1 < \ldots <  b_{j} )   \, |, \, b_i \in \{1, \ldots, n\} \text{ s.t.: }   b_{i-1} \neq b_{i} - 1 \Rightarrow m_{b_i-1} \neq 0\}.
\]
Let $\mathbf{s} \in S_{hw}(\lambda, \omega_j) \subseteq S(\lambda, \omega_j)$, then $s_{k, \ell} = 0$ if $\ell > j $ or $k < j$. We have for all Dyck path $\mathbf{p}$: $\beta_1 + \ldots + \beta_s \leq 1$. This implies that $s_\beta \in \{ 0, 1\}$ for all $\beta$ and even more, that the support of $\mathbf{s}_\alpha$ is of the form
\[
 \{ \alpha_{i_1, j_1} , \ldots, \alpha_{i_\ell, j_\ell} \, | \, i_1 < i_2 \ldots < i_\ell \leq j \leq j_\ell < \ldots < j_1 \}
\]
Let us parametrize this set as follows. Let $ \alpha_{i_1, j_1} , \ldots, \alpha_{i_\ell, j_\ell}$ be given from the set and denote 
\[
 \{ p_1 < \ldots < p_{j-\ell} := \{1, \ldots, j\}\setminus \{ i_1, \ldots, i_\ell \} .
\]
 Then we associate
\[
 \alpha_{i_1, j_1} , \ldots, \alpha_{i_\ell, j_\ell} \leftrightarrow (p_1 < p_2 < \ldots < p_{j-\ell} < j_{\ell} + 1 < \ldots < j_{1} + 1)
\]
This gives a one to one correspondence to $j$-tuples of strictly increasing integers smaller equals to $n$, hence parameterizes a basis of $V(\omega_k)$. \\
Since we are interested in the highest weight vectors, we can exclude these tuples corresponding to vectors in $F_{\lambda, \omega_j}$ of non-dominant weight. The weight of such a vector $(p_1 < p_2 < \ldots < p_j)$  is given by
\[
\lambda + ( \omega_j - ( - \omega_{p_1-1} + \omega_{p_1 } -  \omega_{p_2-1} + \omega_{p_2} -  \ldots - \omega_{p_{j-1}-1} + \omega_{p_{j-1}} \omega_{p_j-1} + \omega_{p_j}).
\]
With a short calculation one sees that this is dominant if and only if $p_i \neq p_{i+1} - 1 \Rightarrow m_{p_i}-1 > 0$.\\
This implies that $a_{\lambda, \omega_j}^{\tau} \leq c_{\lambda, \omega_j}^{\tau}$ for all $\tau \in P^+$. Using Lemma~\ref{whatever} implies now equality for all $\tau$ which proves\begin{lem}
For $\lambda \in P^+$, $j \in \{ 1, \ldots, n-1\} = I$, we have
\[
a_{\lambda, \omega_j}^{\tau} = c_{\lambda, \omega_j}^{\tau} \text{ for all } \tau \in P^+
\]
and so for all $c_1 \neq c_2 \in \bc$
\[
F_{\lambda, \omega_j} \cong V(\lambda)_{c_1}\ast V(\omega_j)_{c_2}.
\]
\end{lem}


\section{Partial order and Weyl modules}\label{poset}
In \cite{CFS14} a partial order on pairs of dominant weight has been introduced. Let us recall here briefly the construction. Fix $\lambda \in P^+$ and consider the partitions of $\lambda$ with two parts
\[
P(\lambda,2) = \{ (\lambda_1, \lambda_2) \in P^+ \times P^+ \, | \, \lambda_1 + \lambda_2 = \lambda \}.
\]
By abuse of notation we denote by $P(\lambda,2)$ the orbits of the natural $S_2$ action on $P(\lambda,2)$. In \cite{CFS14}, the following partial order has been introduced on $P(\lambda,2)$: Let $\blambda = ( \lambda_1, \lambda - \lambda_1), \bmu = (\mu_1, \lambda - \mu_1) \in P(\lambda,2)$, then
\[
\blambda \preceq \bmu : \Leftrightarrow \; \forall \, \alpha \in R^+ : \min\{ \lambda_1(h_\alpha), (\lambda - \lambda_1)(h_\alpha)\} \leq \min \{ \mu_1(h_\alpha), (\mu - \mu_1)(h_\alpha) \}.
\]
Certain properties of this poset were proved in \cite{CFS14} (and \cite{F14b}),  e.g. there exists a smallest element in $P(\lambda,2)$, the orbit of $(\lambda, 0)$. It is less obvious that there exists also a unique maximal element: let $\lambda = \sum_{i = 1}^{n-1} m_i \omega_i$, and let $\{1\leq i_1 < \ldots <  i_k \} =  I_{odd}$ be the indices such that $m_i$ is odd. Then $\blambda^{\max}  = (\lambda_1^{\max}, \lambda_2^{\max})$ given by
\[\lambda_1^{\max}=
  \sum_{j=1}^k ((m_{i_j}+ (-1)^j)/2)\omega_{i_s}+\sum_{i\in I \setminus I_{odd}}(m_i/2)\omega_i, \qquad
  \lambda_2^{\max}=\lambda-\lambda_1^{\max}, 
\]
is the unique maximal orbit in $P(\lambda,2)$, \cite[Proposition 5.3]{CFS14}.\\
It was further shown that the cover relation of $\preceq$ on $P(\lambda,2)$ is determined by the Weyl group action \cite[Proposition 6.1]{CFS14}.

\subsection{}
We want to relate the partial order and the modules $F_{\lambda_1, \lambda_2}$. Namely, we want to prove the following lemma:
\begin{lem}\label{fusion-sp}
Suppose $(\lambda_1, \lambda - \lambda_1) \preceq (\mu_1, \lambda - \mu_1) \in P(\lambda,2)$, then there exists a canonical surjective map of $\lie{sl}_n \otimes \bc[t]$-modules
\[
F_{\mu_1, \mu - \mu_1} \twoheadrightarrow F_{\lambda_1, \lambda - \lambda_1}.
\]
\end{lem}
\begin{proof}
We have to compare the defining relations only. So let $\alpha \in R^+$, then on both modules we have
\[
(f_\alpha \otimes 1)^{\lambda(h_\alpha) + 1}.\mathbbm{1} = 0
\]
and also the highest weight is in both cases $\lambda$. Let $M_1 = \min\{ \mu_1(h_\alpha), \lambda - \mu_1(h_\alpha) \}$ and  $M_2 = \min\{ \lambda_1(h_\alpha), \lambda - \lambda_1(h_\alpha) \}$, then by assumption $M_1 \geq M_2$. By the defining relations of $F_{\lambda_1, \lambda - \lambda_1}$ we have
\[
(f \otimes t)^{M_2 + 1}. \mathbbm{1} = 0 \in F_{\lambda_1, \lambda - \lambda_1}
\]
so especially
\[
(f \otimes t)^{M_1 + 1}. \mathbbm{1} = 0 \in F_{\lambda_1, \lambda - \lambda_1}.
\]
This implies the lemma.
\end{proof}

\subsection{}

We turn to the unique maximal element in $P(\lambda,2)$, $\blambda^{\max} =  (\lambda_1^{\max}, \lambda_2^{\max})$. In fact we want to identify $F_{ \lambda_1^{\max}, \lambda_2^{\max}}$ as the unique graded local Weyl module of $\lie{sl}_n \otimes \bc[t]/(t^2)$ of highest weight $\lambda$. For this we recall the definition of a local Weyl module briefly in the following.\\
Let $A$ be a commutative, finitely generated unital algebra over $\bc$. Then $\lie{sl}_n \otimes A$ is a Lie algebra with bracket given by
\[
[x \otimes p, y \otimes q]= [x,y] \otimes pq
\]
and it is called the generalized current algebra. We fix $\lambda \in P^+$, this induces an one-dimensional $\lie{h}$-modules, which we denote $\bc_{\lambda}$. Let $\xi: (\lie{n^+} \oplus \lie{h}) \otimes A \longrightarrow \lie{h}$ be a Lie algebra homomorphism. Then we can lift the structure on  $\bc_{\lambda}$ to a $ (\lie{n^+} \oplus \lie{h}) \otimes A$-structure, and let us denote this one-dimensional module $\bc_{\lambda, \xi}$. 
\begin{defn}
The local Weyl module $W_A(\xi, \lambda)$ is unique maximal integrable (as a $\lie{sl}_n$-module) quotient of the $\lie{sl}_n \otimes A$-module 
\[
U(\lie{sl}_n \otimes A) \otimes_{ (\lie{n}^+ \oplus \lie{h} ) \otimes A} \bc_{\lambda, \xi}.
\]
\end{defn}
These modules have been introduced for $A=  \bc[t^{\pm 1}]$ in \cite{CP01} and further generalized in \cite{FL04} and \cite{CFK10} to arbitrary commutative associative algebras over $\bc$. It has been shown in \cite{CFK10} that if $A$ is finitely generated, $W_A(\xi, \lambda)$ is finite-dimensional and further that these modules are parameterized by maximal ideals in a tensor product of symmetric powers of $A$.\\
These modules play an important role in the representation theory of $\lie{sl}_n \otimes A$, the interested reader is here referred to \cite{CFK10}.\\ As they are integrable as $\lie{sl}_n$-modules, there exist a decompositions into finite-dimensional simple $\lie{sl}_n$-modules. 
Unfortunately, these decomposition are known for special cases only. Namely for $A= \bc[t], \bc[t^{\pm 1}]$ they are computed in a series on paper \cite{CP01}, \cite{CL06}, \cite{FoL07}. If $A$ is semi-simple, then the local Weyl module obviously decomposes into a direct sum of local Weyl modules for $\lie{sl}_n \otimes \bc = \lie{sl}_n$, so into a direct sum of simple $\lie{sl}_n$-modules. \\
But outside of these cases, even for the ``smallest`` non-semi-simple algebra $A = \bc[t]/(t^2)$, the $\lie{sl}_n$ decomposition is unknown.\\ 
Let us rewrite  the defining relations for the local Weyl modules for $A =  \bc[t]/(t^K)$. In fact, for each $\lambda \in P^+$ and $K\geq 1$, there exists a unique local Weyl module. This follows since there exists a unique non-trivial map $\lambda \circ \xi$, namely $\xi$ is the evaluation map at $t = 0$, so we denote $\xi$ by $0$.
\begin{defn}\label{localweyl}
Let $\lambda \in P^+$, then the \textit{graded local Weyl module} $W_{\bc[t]/(t^K)}(0,\lambda)$ is generated by $w \neq 0$ with relations
\[ (\lie{n}^+ \oplus \lie{h}) \otimes t.w = 0 \;, \;  h - \lambda(h). w = 0  \; , \; \lie n^+.w = 0 \; , \; (f_{\alpha} \otimes 1)^{\lambda(h_\alpha) + 1}.w = 0 .\]
\end{defn}
Since the relations are homogeneous, we see that $W_{\bc[t]/(t^K)}(0,\lambda)$ is a graded $\lie{sl}_n \otimes \bc[t]/(t^K)$-module. Even more, we have immediately from the defining relations 
\begin{prop}\label{f-weyl}
Let $\lambda_1 + \lambda _2 = \lambda \in P^+$ and $K \geq 2$, then there exists a surjective map of  $\lie{sl}_n \otimes \bc[t]$-modules
\[
W_{\bc[t]/(t^K)}(0,\lambda) \twoheadrightarrow F_{\lambda_1, \lambda_2}.
\]

\end{prop}
In fact $ F_{\lambda_1, \lambda_2}$ is the quotient obtained by factorizing the $U( \lie{sl}_n \otimes \bc[t])$-submodule generated by 
\[
\{ (f_\alpha \otimes t)^{\min\{ \lambda_1(h_\alpha), \lambda_2(h_\alpha) \} +1 }.\mathbbm{1} \, | \, \alpha \in R^+ \} \cup \{ f_\alpha \otimes t^{\ell} \, | \, \ell \geq 2, \alpha \in R^+ \}.
\]

\subsection{}
In this subsection we are restricting ourselves to the case of the second truncated current algebra, and we denote $A = \bc[t]/(t^2)$. We will prove
\begin{lem}\label{weyl-f} Let $\lambda \in P^+$ and $\blambda^{\max} =  (\lambda_1^{\max}, \lambda_2^{\max})$ be the unique maximal element in $P(\lambda, 2)$. Then we have an isomorphism of $\lie{sl}_n \otimes A$-modules (and by extending an isomorphism of $\lie{sl}_n \otimes \bc[t]$-modules):
\[
 W_A(0,\lambda) \cong F_{ \lambda_1^{\max}, \lambda_2^{\max}}.
\]
\end{lem}
\begin{proof} 
We consider the $\lie{sl}_2$-case first. Then $\lambda = m \omega$ and because $e, e \otimes t, h \otimes t$ are acting trivial on $\mathbbm{1}$,
\[
W_A(0,\lambda) = \operatorname{span} \{ f^K (f \otimes t)^L .w \,| ,\ K, L \geq 0 \}
\]
So if we restrict to elements in degree $L$ (recall, that $W_A(0,\lambda)$ is graded by the degree of $t$), then this is spanned by
\[
\{  f^K (f \otimes t)^L.w \, | \, K \geq 0 \}
\]
The weights in degree $L$ are therefore of the form $m - 2L - 2K$ with $K \geq 0$. Every graded component is a finite-dimensional $\lie{sl}_2$-module, since $\lie{sl}_2$ acts by degree $0$ and $W_A(0,\lambda)$ is finite-dimensional. This implies that the component of degree $L$ in $W(0,\lambda)$ is $0$ if there is no vector of dominant weight in degree $L$. So for $L > \lfloor m/2 \rfloor$ we have $(f  \otimes t)^L.w = 0$. \\
On the other hand, $\blambda = ( \lfloor m/2 \rfloor, \lceil m/2 \rceil )$ which implies that  $ F_{\lambda_1^{\max}, \lambda_2^{\max} }$ is the quotient by the submodule generated by 
\[
(f \otimes t)^L.\mathbbm{1}  \text{ with } L > \lfloor m/2 \rfloor.
\]
This implies that $W_A(0,\lambda) \cong F_{\lambda_1^{\max}, \lambda_2^{\max} }$.\\
Let us turn to the general case. We have 
\[
\min \{ \lambda_1^{\max}(h_\alpha), \lambda_2^{\max}(h_\alpha) \}  = \lfloor \lambda(h_\alpha)/2 \rfloor.
\]  
It is enough to show that $(f_\alpha \otimes t)^{ \lfloor \lambda(h_\alpha)/2 \rfloor + 1}.\mathbbm{1} = 0 \in W_A(0, \lambda)$ for all $\alpha$. \\
Fix $\alpha > 0$ and consider the Lie subalgebra $\lie{sl}(\alpha) \otimes A =  \langle e_\alpha, h_\alpha, f_\alpha, e_\alpha \otimes t, h_\alpha \otimes t, f_\alpha \otimes t \rangle$ which is isomorphic to $\lie{sl}_2 \otimes A$.\\
We consider the submodule $M = U(\lie{sl}(\alpha) \otimes A).\mathbbm{1} \subseteq W_A(0,\lambda)$. Then this is a quotient of the $\lie{sl}_2 \otimes A$ local Weyl module $W_A(0, \lambda(h_\alpha) \omega)$ (since the defining relations are satisfied on the highest weight vector).\\
The considerations above for the $\lie{sl}_2$-case imply now that 
\[
(f_\alpha \otimes t)^{\lfloor \lambda(h_\alpha)/2 \rfloor + 1}.\mathbbm{1} = 0  \in M \subseteq W_A(0,\lambda)
\]
Which implies that $W_A(0, \lambda)$ is a quotient of $F_{\lambda_1^{\max}, \lambda_2^{\max}}$ and hence they are isomorphic.
\end{proof}

\bibliographystyle{alpha}
\bibliography{biblist-fusion}

\end{document}